\documentclass[12pt,a4paper]{article}
\usepackage{amsmath}
\usepackage{amssymb}
\usepackage{amsthm}

\usepackage{verbatim}

\usepackage[utf8]{inputenc}%(only for the pdftex engine)
\usepackage{graphicx}
\usepackage{color}
\usepackage{enumerate}

\usepackage{esint}
\usepackage[all]{xy}
\usepackage{framed}

\theoremstyle{plain}
\newtheorem{theorem}{Theorem}

\newtheorem{lemma}{Lemma}
\newtheorem{proposition}{Proposition}
\newtheorem*{proposition_uten_num}{Proposition}

\theoremstyle{definition}
\newtheorem{definition}{Definition}
\numberwithin{equation}{section}

\newcommand{\bfx}{\mathbf{x}}

\newcommand{\bfc}{\mathbf{c}}

\newcommand{\bfe}{\mathbf{e}}
\newcommand{\Hu}{\mathcal{H}u}
\newcommand{\D}{\mathcal{D}}
\newcommand{\hu}{\widehat{\nabla}u}

\newcommand{\dd}{\mathrm{\,d}}

\newcommand{\od}{\Delta^N}
\DeclareMathOperator{\tr}{tr}

\DeclareMathOperator{\di}{div}

\begin{document}
\author{Karl K. Brustad\\ {\small Norwegian University of Science and Technology}}
\title{Superposition of p-superharmonic functions}
\maketitle

\begin{abstract}
The \emph{Dominative} \(p\)-Laplace Operator is introduced. This operator is a relative to the \(p\)-Laplacian, but with the distinguishing property of being sublinear. It explains the
superposition principle in the \(p\)-Laplace Equation.
\end{abstract}
\section{Introduction}

The \(p\)-Laplace Equation
\begin{equation}
\Delta_p u := \di(|\nabla u|^{p-2}\nabla u) = 0
\label{ple}
\end{equation}
is the Euler-Lagrange equation for the variational integral \(\int_\Omega|\nabla u|^p\dd x\).
When \(p=2\) we have Dirichlet's integral and Laplace's Equation \(\Delta u = 0\). For \(p=\infty\) we define
\begin{equation}
\Delta_\infty u := \sum_{i,j=1}^n\frac{\partial^2 u}{\partial x_i\partial x_j}\frac{\partial u}{\partial x_i}\frac{\partial u}{\partial x_j} = 0.
\label{infop}
\end{equation}
The object of our work is a superposition principle, originally discovered by Crandall and Zhang in \cite{MR1928087}.

Although the \(p\)-Laplace Equation is nonlinear when \(p\neq 2\), a principle of superposition for the \emph{fundamental solutions}
\begin{equation}
w_{n,p}(x) :=
\begin{cases}
-\frac{p-1}{p-n}|x|^\frac{p-n}{p-1}, & p\neq n,\\
-\ln|x|, & p = n,\\
-|x|, & p=\infty\;\text{ or }\; n=1,
\end{cases}
\label{fundsol}
\end{equation}
is valid. That is, if \(p\geq 2\) and \(c_i\geq 0\), then the superposition
\begin{equation}
V(x) := \sum_{i=1}^N c_i w_{n,p}(x-y_i)
\label{crand}
\end{equation}
satisfies \(\Delta_p V\leq 0\) away from the poles \(y_1,\dots,y_N\). Moreover, the sum is \emph{\(p\)-superharmonic} in the whole of \(\mathbb{R}^n\) according to Definition \ref{visdef} in Section \ref{sect_vis}.
In \cite{Brustad201723} an explicit formula for \(\Delta_p V(x)\) was derived.
There it was shown that an arbitrary concave function also may be added to the sum \eqref{crand}.
The result can be extended to infinite sums, and via Riemann sums one obtains that the potentials
\[V(x) = \int_{\mathbb{R}^n}\rho(y)w_{n,p}(x-y)\dd y,\qquad \rho\geq 0,\]
are \(p\)-superharmonic functions, provided that \(V(x) \not\equiv \infty\).
See \cite{MR2350398} and \cite{MR2914612}.

It has been a little mystery \emph{why} the sum \eqref{crand} is \(p\)-superharmonic. It has not been clear what the underlying reasons are, or how far the superposition could be extended.
It turns out that a class of functions called \emph{dominative \(p\)-superharmonic functions} plays a central rôle in these questions. We introduce them through the sublinear operator
\begin{equation}
\D_p u := \lambda_1 + \cdots + \lambda_{n-1} + (p-1)\lambda_n,\qquad p\geq 2,
\label{predef}
\end{equation}
where \(\lambda_1\leq\cdots\leq\lambda_n\) are the eigenvalues of the Hessian matrix
\[\Hu := \left(\frac{\partial^2 u}{\partial x_i\partial x_j}\right)_{i,j = 1}^n.\]

The fundamental solutions are members of this class and for \(C^2\) functions we have
\begin{proposition_uten_num}
Let \(u\in C^2(\Omega)\). Then
\[\D_p u\leq 0\qquad\Rightarrow\qquad \Delta_p u\leq 0\qquad\text{in \(\Omega\).}\]
\end{proposition_uten_num}
In general, the inequality \(\D_pu\leq 0\) must be interpreted in \emph{the viscosity sense}, see Section \ref{sect_vis}. As we shall see, the superposition principle is governed by the equation \(\D_pu\leq 0\).

Needless to say, the eigenvalues of \(\Hu\) have been much studied.
In \cite{MR2817413} the equation \(\lambda_1 = 0\) is found.
The related equation \(\lambda_n = 0\) is produced by the dominative operator in the limit \(p\to\infty\):
\[\tfrac{1}{p}\D_p u \to \lambda_n =: \D_\infty u.\]
The supersolutions \(\D_\infty u\leq 0\) are, in fact, the \emph{concave} functions. We also mention the paper \cite{MR2128299} where symmetric functions of the eigenvalues are investigated.
So far as we know, the Dominative \(p\)-Laplace Equation is new for \(p\neq\infty\).

Our main results are Theorem \ref{mainthm1} and Theorem \ref{mainthm2} below. Theorem \ref{mainthm1} gives sufficient and necessary conditions for a sum to be \(p\)-superharmonic.
In short: a generic sum is \(p\)-superharmonic if and only if its terms are dominative \(p\)-superharmonic functions. Furthermore, it will be demonstrated that
one cannot expect a sum of \(p\)-\emph{harmonic} functions (like \eqref{crand}) to be \(p\)-superharmonic unless the functions involved have a high degree of symmetry (Definition \ref{cyldef}).

In Theorem \ref{mainthm2} we extend the superposition principle for the fundamental solutions to \emph{arbitrary radial \(p\)-superharmonic functions.} Its proof is obtained by showing that important properties of the dominative \(p\)-Laplace operator in the smooth case, also hold in the viscosity sense.

Throughout the paper, we restrict ourselves to the case \(2\leq p\leq\infty\) and \(n\geq 2\). An open subset of \(\mathbb{R}^n\) is denoted by \(\Omega\).

\begin{theorem}\label{mainthm1}
Let \(2\leq p\leq\infty\). The following conditions hold pointwise in \(\mathbb{R}^n\).
\begin{enumerate}[(i)]
	\item Let \(u_1,\dots,u_N\) be dominative \(p\)-superharmonic \(C^2\) functions. Then
	\[\Delta_p\left[\sum_{i=1}^N u_i\right] \leq 0.\]
	\item Let \(u\) be \(C^2\). Then the following claims are equivalent.
	\begin{enumerate}
		\item For every linear function \(l(x) = a^Tx\),\[\Delta_p[u+l] \leq 0.\]
		\item For all constants \(c\geq 0\) and every translation \(T(x) = x-x_0\),\[\Delta_p[u+c\,w_{n,p}\circ T] \leq 0.\]
		\item For every isometry \(T\colon\mathbb{R}^n\to\mathbb{R}^n\),\[\Delta_p\left[u+u\circ T\right] \leq 0.\]
		\item \(u\) is a dominative \(p\)-superharmonic function.
	\end{enumerate}
If \(u\), in addition, is \(p\)-harmonic and \(2<p<\infty\), then
	\begin{enumerate}\setcounter{enumii}{4}
		\item \(u\) is locally a cylindrical fundamental solution (see Def. \ref{cyldef}).
	\end{enumerate}
\end{enumerate}
\end{theorem}

\begin{theorem}\label{mainthm2}
Let \(2\leq p\leq\infty\) and let \(u_1,\dots,u_N\) be radial \(p\)-superharmonic functions in \(\mathbb{R}^n\). Then the sum
\[\sum_{i=1}^N u_i(x-y_i) + K(x),\qquad y_i\in\mathbb{R}^n,\]
is \(p\)-superharmonic in \(\mathbb{R}^n\) for any concave function \(K\).
\end{theorem}

To keep things simple, we have not extended Theorem \ref{mainthm2} to cover \emph{cylindrical} \(p\)-superharmonic functions.

A function \(f\) in \(\mathbb{R}^n\) is \emph{radial} if there exists a one-variable function \(F\) so that \(f(x) = F(|x|)\). As usual, \(|x| := \sqrt{x_1^2 + \cdots + x_n^2}\) denotes the Euclidean norm of \(x\). An equivalent definition is the symmetry condition \(f(Qx) = f(x)\) for every \(n\times n\) orthogonal matrix \(Q\). Radial functions have concentric spherical level-sets, and so do the \emph{translated} ones \(f(x-x_0)\). We generalize this notion to functions having concentric \emph{cylindrical} level-sets:

\begin{definition}\label{cyldef}
A function \(f\) in \(\mathbb{R}^n\) is \textbf{cylindrical} (or \(k\)-cylindrical) if there exists a one-variable function \(F\), an integer \(1\leq k\leq n,\) and an \(n\times k\) matrix \(Q\) with orthonormal columns, i.e. \(Q^TQ = I_k\), so that
\[f(x) = F\Big(|Q^T(x-x_0)|\Big)\]
for some \(x_0\in\mathbb{R}^n\).

We say that a function \(u\) in \(\mathbb{R}^n\) is a \textbf{cylindrical fundamental solution} (to the \(p\)-Laplace Equation) if \(u\) is on the form
\[u(x) = C_1w_{k,p}\Big(Q^T(x-x_0)\Big) + C_2,\qquad C_1\geq 0,\]
for some \(k\), \(Q\) and \(x_0\) as above.
\end{definition}

Notice that a 1-cylindrical fundamental solution
\[u(x) = -C_1|q^T(x-x_0)| + C_2 = a^Tx + b\]
is \emph{affine} in the regions where it is differentiable, while an \(n\)-cylindrical fundamental solution
\[u(x) = C_1w_{n,p}\Big(Q^T(x-x_0)\Big) + C_2 = C_1w_{n,p}(x-x_0) + C_2\]
is a translated radial function.

A calculation shows that the cylindrical fundamental solutions solve the \(p\)-Laplace Equation, except on the \((n-k)\)-dimensional affine space \(\left\{x\;|\; Q^T(x-x_0) = 0\right\}\) where the functions become singular. When \(p<\infty\) a Dirac delta is produced. For example, setting \(Q = (\bfe_1,\dots,\bfe_k)\), \(x_0=0\) and splitting \(x = (y,z)\in\mathbb{R}^n\) in \(y\in\mathbb{R}^k\), \(z\in\mathbb{R}^{n-k}\) yields
\begin{align*}
\int_{\mathbb{R}^n}\Delta_p w_{k,p}\big(Q^Tx\big)\phi(x)\dd x
	&= \int_{\mathbb{R}^{n-k}}\int_{\mathbb{R}^k}\Delta_p w_{k,p}(y)\phi(y,z)\dd y\dd z\\
	&= -C\int_{\mathbb{R}^{n-k}}\phi(0,z)\dd z,\qquad C = C(k,p) > 0,
\end{align*}
for every \(\phi\in C^\infty_0(\mathbb{R}^n)\).
Moreover, we shall see that these solutions have an essential property that is not shared by any other \(p\)-harmonic function:
\begin{framed}\noindent
The gradient of a cylindrical fundamental solution is an eigenvector corresponding to the largest eigenvalue of the Hessian matrix.
\end{framed}

\section{The Dominative p-Laplace Operator}

\subsection{Preliminary basics and notation}

For a function \(u\in C^2(\Omega)\), we denote by \(\mathcal{H}u = \mathcal{H}u(x)\) the Hessian matrix of \(u\) at \(x\).
We list some elementary and useful facts about this matrix.
\begin{itemize}
	\item \(\mathcal{H}u\) is a symmetric \(n\times n\) matrix: \(\mathcal{H}u^T = \mathcal{H}u\).
	\item \(\mathcal{H}u\) has \(n\) real eigenvalues, which we label in increasing order:
	\[\lambda_1\leq\dots\leq\lambda_n.\]
	The \emph{largest} eigenvalue, \(\lambda_n\), has special importance and is denoted by \(\lambda_u\) to indicate its origin. Sometimes we are inconsistent with the notation and write \(\lambda_X\) for the largest eigenvalue of a symmetric matrix \(X\).
	\item The eigenvectors, \(\xi_1,\dots,\xi_n\), of \(\mathcal{H}u\) can be chosen to be orthonormal: \(\xi_i^T\xi_j = \delta_{ij}.\)
	A unit eigenvector corresponding to the largest eigenvalue \(\lambda_u\) is labeled \(\xi_u\).
	\item \(\tr\mathcal{H}u = u_{x_1x_1}+\cdots+u_{x_nx_n} = \lambda_1 +\cdots +\lambda_n = \Delta u\).
	In general
	\[\Delta u = \sum_{i=1}^n z_i^T\mathcal{H}u\,z_i\]
	for every orthonormal set \(\{z_1,\dots,z_n\}\subseteq\mathbb{R}^n\).
	\item We have for any vector \(z\in\mathbb{R}^n\),
	\[\lambda_1|z|^2 \leq z^T\mathcal{H}u\,z \leq \lambda_u|z|^2,\]
	and
	\[\lambda_u = \max_{|z|=1}z^T\mathcal{H}u\,z.\]
	Conversely, if \(z\in\mathbb{R}^n\), \(|z|=1\) satisfies
	\[z^T\mathcal{H}u\,z = \lambda_u,\]
	then
	\[\mathcal{H}u\,z = \lambda_u z.\]
\end{itemize}
We adapt the convention that vectors/points in space are column vectors, \emph{except} gradients which are to be read as row vectors. 

\subsection{Definition and fundamental properties}

\begin{definition}\label{lpdef}
We define the \textbf{Dominative\footnote{The reader interested in \(p\)-\emph{sub}harmonic functions should consider the operator \[-\mathcal{D}_p[-u] = (p-2)\lambda_1 + \Delta u.\] Dare we suggest the name ``the Submissive \(p\)-Laplace Operator, \(\mathcal{S}_p\)''?} \(p\)-Laplace Operator}, \(\mathcal{D}_p\), as
\[\mathcal{D}_p u :=
\begin{cases}
(p-2)\lambda_u + \Delta u,\qquad & \text{when } 2\leq p <\infty,\\
\lambda_u, &\text{when }p =\infty.
\end{cases}\]
A \(C^2\) function \(u\) is \textbf{dominative \(p\)-superharmonic} if \(\D_pu\leq 0\) at each point in its domain.
\end{definition}
The expression \eqref{predef} is, of course, an alternative representation when \(p<\infty\).
Observe that \(\D_2 = \Delta_2 = \Delta\).

In low dimensions it is possible to express \(\mathcal{D}_pu\) in terms of the second-order partial derivatives \(u_{x_ix_j}\). In \(\mathbb{R}^2\) it can be calculated to be
\begin{align*}
\mathcal{D}_p u &= \frac{p}{2}(u_{xx} + u_{yy}) + \frac{p-2}{2}\sqrt{(u_{xx}-u_{yy})^2 + 4u_{xy}^2},\\
\mathcal{D}_\infty u &= \frac{1}{2}(u_{xx} + u_{yy}) + \frac{1}{2}\sqrt{(u_{xx}-u_{yy})^2 + 4u_{xy}^2}.
\end{align*}
We clearly see the nonlinearity introduced when \(p>2\).

The motivation behind Definition \ref{lpdef} came from the following observations:
By carrying out the differentiation in \eqref{ple}, we arrive at the identity
\begin{equation}
\Delta_p u = |\nabla u|^{p-2}\left((p-2)\frac{\nabla u\mathcal{H}u\nabla u^T}{|\nabla u|^2} + \Delta u\right).
\label{lpder}
\end{equation}
The normalized variant of the \(\infty\)-Laplacian appearing in \eqref{lpder} satisfies
\begin{equation}
\od_\infty u := \frac{\Delta_\infty u}{|\nabla u|^2} = \frac{\nabla u\mathcal{H}u\nabla u^T}{|\nabla u|^2} \leq \lambda_u = \mathcal{D}_\infty u.
\label{infnorm}
\end{equation}
Thus the normalized \(p\)-Laplacian also satisfies
\begin{equation}
\od_p u := \frac{\Delta_pu}{|\nabla u|^{p-2}} = (p-2)\od_\infty u + \Delta u \leq (p-2)\lambda_u + \Delta u = \mathcal{D}_p u
\label{pnorm}
\end{equation}
when \(0\leq p-2<\infty\).
When \(u\) is a fundamental solution we have equality in \eqref{infnorm} and \eqref{pnorm}. Since \(\mathcal{D}_p\) is \emph{sublinear} and invariant under translations, a very simple proof of the superposition principle for the fundamental solutions \eqref{crand} is produced. However, the above calculations are, for the moment, not valid at the poles or at critical points.

\begin{proposition}[Fundamental properties of \(\mathcal{D}_p\). Smooth case]\label{D_p prop}
Let \(u,v\in C^2(\Omega)\). Then the following holds pointwise for \(2\leq p\leq\infty\).
\begin{enumerate}[(1)]
	\item \textbf{Domination:}
	\[\Delta_pu \leq
\begin{cases}
|\nabla u|^{p-2}\mathcal{D}_p u,\qquad &2\leq p<\infty,\\
|\nabla u|^2 \mathcal{D}_\infty u, &p = \infty.
\end{cases}
\]
	\item \textbf{Sublinearity:}
	\begin{itemize}
		\item \(\mathcal{D}_p[u+v] \leq \mathcal{D}_p u + \mathcal{D}_p v\), and
		\item \(\mathcal{D}_p[\alpha u] = \alpha \mathcal{D}_p u,\qquad \alpha\geq 0\).
	\end{itemize}
	\item \textbf{Cylindrical Equivalence:} Assume \(u\) is \(k\)-cylindrical where the corresponding one-variable function \(U=U(r)\) satisfies \(\frac{U'}{r}\leq U''\). If \(k<n\), we also require that \(U''\geq 0\). Then
	\begin{equation}
	\Delta_pu =
\begin{cases}
|\nabla u|^{p-2}\mathcal{D}_p u,\qquad &2\leq p<\infty,\\
|\nabla u|^2 \mathcal{D}_\infty u, &p = \infty.
\end{cases}
	\label{cyleq}
	\end{equation}
	In particular, if \(u\) is a cylindrical fundamental solution, then \[\mathcal{D}_p u = 0 = \Delta_p u.\]
Moreover, if \(u\) is \(k\)-cylindrical and \(\mathcal{D}_p u = 0 = \Delta_p u\), then \(u\) is a \(k\)-cylindrical fundamental solution provided \(2<p<\infty\) and \(k\geq 2\).
	\item \textbf{Nesting Property:}
	\begin{itemize}
		\item If \(\mathcal{D}_p u\leq 0\), then \(\mathcal{D}_q u\leq 0\) for every \(2\leq q\leq p\).
		\item If \(\mathcal{D}_p u\geq 0\), then \(\mathcal{D}_q u\geq 0\) for every \(p\leq q\leq\infty\).
	\end{itemize}
	\item \textbf{Invariance:} \[\mathcal{D}_p [u\circ T] = (\mathcal{D}_p u)\circ T\] in \(T^{-1}(\Omega)\) for all isometries \(T\colon \mathbb{R}^n\to\mathbb{R}^n\).
\end{enumerate}
\end{proposition}

\begin{proof}\textbf{Domination:}\\
Since
\[\Delta_\infty u = \nabla u\mathcal{H}u\nabla u^T \leq |\nabla u|^2\lambda_u = |\nabla u|^2\mathcal{D}_\infty u,\]
we also get, when \(\nabla u \neq 0\),
\begin{align*}
\Delta_p u &= |\nabla u|^{p-2}\left((p-2)\od_\infty u + \Delta u\right)\\
				&\leq |\nabla u|^{p-2}\left((p-2)\lambda_u + \Delta u\right)\\
				&= |\nabla u|^{p-2}\mathcal{D}_p u
\end{align*}
for \(0\leq p-2<\infty\). If \(\nabla u = 0\) or \(p=2\), the claim is trivial: The \(p\)-Laplacian is zero at critical points when \(p>2\).
\end{proof}

\begin{proof}\textbf{Sublinearity:}\\
Since
\begin{align*}
\mathcal{D}_\infty[u+v] &= \lambda_{u+v}\\
              &= \xi_{u+v}^T\mathcal{H}[u+v]\xi_{u+v}\\
				      &= \xi_{u+v}^T\big(\mathcal{H}u + \mathcal{H}v\Big)\xi_{u+v}\\
				      &= \xi_{u+v}^T\mathcal{H}u\xi_{u+v} + \xi_{u+v}^T\mathcal{H}v\xi_{u+v}\\
				      &\leq \lambda_u + \lambda_v\\
							&= \mathcal{D}_\infty u + \mathcal{D}_\infty v
\end{align*}
we also get
\begin{align*}
\mathcal{D}_p[u+v] &= (p-2)\lambda_{u+v} + \Delta[u+v]\\
				 &\leq (p-2)(\lambda_u + \lambda_v) + \Delta u + \Delta v\\
				 &= \mathcal{D}_p u + \mathcal{D}_p v
\end{align*}
for \(0\leq p-2<\infty\).

Also, if \(\alpha\geq 0\) and \(\lambda_u\) is the largest eigenvalue of \(\mathcal{H}u\), then \(\alpha\lambda_u\) is the largest eigenvalue of \(\mathcal{H}[\alpha u] = \alpha\mathcal{H}u\). Thus \(\lambda_{\alpha u} = \alpha\lambda_u\). This means that
\[\mathcal{D}_\infty[\alpha u] = \lambda_{\alpha u} = \alpha\lambda_u = \alpha \mathcal{D}_\infty u\]
and
\[\mathcal{D}_p[\alpha u] = (p-2)\lambda_{\alpha u} + \Delta[\alpha u] = \alpha \mathcal{D}_p u\]
for \(p<\infty\).
\end{proof}

\begin{proof}\textbf{Cylindrical Equivalence:}\\
Let \(1\leq k\leq n\), \(x_0\in\mathbb{R}^n\), let \(Q\) be an \(n\times k\) matrix with \(Q^TQ = I_k\), and let
\[u(x) = U\Big(|Q^T(x-x_0)|\Big).\]
Write \(y\colon\Omega\subseteq\mathbb{R}^n\to\mathbb{R}^k\),
\[y(x) := Q^T(x-x_0).\]
Assume first that \(u\) is \(C^2\) at a point \(x_1\in\Omega\) where \(y(x_1) = 0\). A direct calculation will confirm that
\[\nabla u(x_1) q = - \nabla u(x_1) q\qquad\forall\,q\in\mathbb{R}^n.\]
Therefore, \(\nabla u(x_1) = 0\) and the equality in \eqref{cyleq} is trivial.

The Jacobian matrix of \(y\) is \(Dy = Q^T\), and
\[\nabla|y| = \frac{y^T}{|y|}Dy = \hat{y}^TQ^T,\qquad \hat{y} := \frac{y}{|y|}.\]
Moreover,
\[\nabla\frac{1}{|y|} = -\frac{y^T}{|y|^3}Dy\]
so
\begin{align*}
D\hat{y} &= \frac{Dy}{|y|} + y\nabla\frac{1}{|y|}\\
         &= \frac{1}{|y|}\left(Dy - y\frac{y^T}{|y|^2}Dy\right)\\
				 &= \frac{1}{|y|}\left(Q^T - \hat{y}\hat{y}^TQ^T\right).
\end{align*}
Now, \(u(x) = U(|y(x)|)\) and
\begin{align*}
\nabla u &= U'\nabla|y|\\
         &= U'\hat{y}^TQ^T.
\end{align*}
The Hessian matrix is
\begin{align*}
\mathcal{H}u &= D\left[\nabla u^T\right]\\
             &= D\left[U'Q\hat{y}\right]\\
						 &= Q\hat{y}U''\nabla|y| + U'Q D\hat{y}\\
						 &= U''Q\hat{y}\hat{y}^TQ^T + \frac{U'}{|y|}Q\left(Q^T - \hat{y}\hat{y}^TQ^T\right)\\
						 &= Q\left\{U''\hat{y}\hat{y}^T + \frac{U'}{|y|}\left(I_k - \hat{y}\hat{y}^T\right)\right\}Q^T
\end{align*}
with trace
\[\Delta u = U'' + (k-1)\frac{U'}{|y|}.\]
There are \(n-k\) perpendicular constant eigenvectors in the null-space of \(Q^T\) with zero eigenvalues. The (transposed) gradient is in the column-space of \(Q\) and is an eigenvector:
\begin{align*}
\mathcal{H}u\nabla u^T &= Q\left\{U''\hat{y}\hat{y}^T + \frac{U'}{|y|}\left(I_k - \hat{y}\hat{y}^T\right)\right\}Q^T(U'Q\hat{y})\\
	&= U'Q\left\{U''\hat{y}\hat{y}^T + \frac{U'}{|y|}\left(I_k - \hat{y}\hat{y}^T\right)\right\}\hat{y}\\
	&= U'Q\Big\{U''\hat{y} + 0\Big\}\\
	&= U''\nabla u^T.
\end{align*}
Finally there are \(k-1\) eigenvectors \(\xi = \xi(x) \in\mathbb{R}^n\) in the column-space of \(Q\) that are perpendicular to \(\nabla u\), i.e.
\[\hat{y}^TQ^T\xi=0\qquad\text{and}\qquad \xi = Q\tilde{\xi}\quad\text{for some \(\tilde{\xi}\in\mathbb{R}^k\):}\]
\begin{align*}
\mathcal{H}u\xi &= Q\left\{U''\hat{y}\hat{y}^T + \frac{U'}{|y|}\left(I_k - \hat{y}\hat{y}^T\right)\right\}Q^T\xi\\
                &= Q\left\{0 + \frac{U'}{|y|}(Q^T\xi - 0)\right\}\\
								&= \frac{U'}{|y|}QQ^T\xi\\
								&= \frac{U'}{|y|}QQ^TQ\tilde{\xi} = \frac{U'}{|y|}Q\tilde{\xi} = \frac{U'}{|y|}\xi.
\end{align*}
Thus the \(n\) eigenvalues of \(\mathcal{H}u\) are
\(\frac{U'}{|y|}\) with multiplicity \(k-1\), \(0\) with multiplicity \(n-k\) and \(U''\) with multiplicity \(1\).

By the assumption \(\frac{U'}{r}\leq U''\) and \(U''\geq 0\) if \(k<n\) it is clear that the largest eigenvalue is \(\lambda_u = U''\) and it follows that
\[\Delta_\infty u = \nabla u\mathcal{H}u\nabla u = |\nabla u|^2 U'' = |\nabla u|^2\mathcal{D}_\infty u\]
and
\begin{align*}
\Delta_p u &= |\nabla u|^{p-2}\left((p-2)\frac{\nabla u\mathcal{H}u\nabla u^T}{|\nabla u|^2} + \Delta u\right),\qquad\nabla u\neq 0,\\
           &= |\nabla u|^{p-2}\left((p-2)\lambda_u + \Delta u\right)\\
					 &= |\nabla u|^{p-2}\mathcal{D}_p u.
\end{align*}
Again, the equality is trivial if \(\nabla u = 0\).

Now assume \(u(x) = U(|y|)\) is a \(C^2\) \(k\)-cylindrical fundamental solution:
\[U(r) = 
\begin{cases}
-C_1\frac{p-1}{p-k}r^\frac{p-k}{p-1} + C_2, & p\neq k,\\
-C_1\ln r + C_2, & p = k,\\
-C_1r + C_2, & p=\infty\;\text{ or }\; k=1,
\end{cases}\]
\(C_1\geq 0\), \(C_2\in\mathbb{R}\).
Then
\[U'(r) = 
\begin{cases}
-C_1r^\frac{1-k}{p-1}, & 2\leq p<\infty,\\
-C_1, & p=\infty
\end{cases}\]
and
\[U''(r) = 
\begin{cases}
-C_1\frac{1-k}{p-1}r^\frac{2-p-k}{p-1}, & 2\leq p<\infty,\\
0, & p=\infty.
\end{cases}\]
We see that \(U'' \geq 0 \geq U'/r\) in every case, so \(\mathcal{D}_\infty u = \lambda_u = U'' = 0\) if \(p=\infty\), and
\[\Delta_\infty u = |\nabla u|^2\mathcal{D}_\infty u = 0.\]
Also,
\begin{align*}
\mathcal{D}_p u &= (p-2)U'' + U'' + (k-1)\frac{U'}{|y|}\\
  		&= (p-1)U'' + (k-1)\frac{U'}{|y|}\\
			&= -C_1\left((1-k)|y|^\frac{2-p-k}{p-1} + (k-1)|y|^{\frac{1-k}{p-1}-1}\right)\\
			&= 0
\end{align*}
and
\[\Delta_p u = |\nabla u|^{p-2}\mathcal{D}_p u = 0\]
when \(p<\infty\).

Finally, assume \(u\) is \(k\)-cylindrical, \(2\leq k\leq n\), \(2<p<\infty\) and \(\Delta_p u = 0 = \D_p u\).
The ODE for \(U\) produced by \(\Delta_p u = 0\) is
\begin{equation}
0 = \Delta_p^N u = (p-1)U'' + (k-1)\frac{U'}{r}
\label{Lpode}
\end{equation}
with general solution \(U\) satisfying
\[U'(r) = C r^{-\frac{k-1}{p-1}},\qquad C\in\mathbb{R}.\]
By definition of the cylindrical fundamental solutions, we only need to show that \(C\leq 0\). The equation \(\D_p u = 0\) gives
\begin{align}
0 &= (p-2)\max\left\{U'',\frac{U'}{r}\right\} + U'' + (k-1)\frac{U'}{r}, &&k = n,\label{Dpn}\\
0 &= (p-2)\max\left\{U'',\frac{U'}{r},0\right\} + U'' + (k-1)\frac{U'}{r}, &&2\leq k < n\label{Dpk}.
\end{align}
Subtract \eqref{Lpode} from \eqref{Dpn} or \eqref{Dpk} and 
divide by \(p-2\) to obtain the condition
\[0 = \max\left\{U'',\frac{U'}{r}\right\} - U''.\]
That is
\[-C\frac{k-1}{p-1}r^{-\frac{k+p-2}{p-1}} = U'' \geq \frac{U'}{r} = Cr^{-\frac{k+p-2}{p-1}}\]
which is true only if \(C\leq 0\).
\end{proof}

\begin{proof}\textbf{Nesting Property:}\\
Let \(2\leq q\leq p \leq\infty\) and assume \(\mathcal{D}_p u\leq 0\). For each \(x\in\Omega\) we consider two cases.

If \(\mathcal{D}_\infty u = \lambda_u < 0\) then every eigenvalue is negative and
\[\mathcal{D}_q u = \lambda_1 + \cdots + \lambda_{n-1} + (q-1)\lambda_u < 0.\]
If \(\mathcal{D}_\infty u = \lambda_u \geq 0\) then also
\begin{align*}
\mathcal{D}_q u &= (q-2)\lambda_u + \Delta u\\
      &\leq (p-2)\lambda_u + \Delta u\\
			&= \mathcal{D}_p u \leq 0.
\end{align*}

Now let \(2\leq p\leq q\leq\infty\) and assume
\[\mathcal{D}_p u = (p-1)\lambda_u + \lambda_1 + \cdots + \lambda_{n-1}\geq 0.\]
Then, obviously \(\lambda_u\geq 0\) and
\begin{align*}
\mathcal{D}_q u &= (q-2)\lambda_u + \Delta u\\
      &\geq (p-2)\lambda_u + \Delta u\\
			&= \mathcal{D}_p u \geq 0.
\end{align*}

\end{proof}

\begin{proof}\textbf{Invariance:}\\
An isometry in \(\mathbb{R}^n\) is on the form \(T(x) = Q^T(x-x_0)\) for some \(x_0\in\mathbb{R}^n\) and some constant orthogonal \(n\times n\) matrix \(Q\): \(Q^TQ = I\).

Define \(v(x) := u(T(x))\) on \(T^{-1}(\Omega)\), i.e. \(T(x)\in\Omega\). Write \(y := T(x)\). Then \(\nabla v(x) = \nabla u(y)Q^T\) and
\[\mathcal{H}v(x) = Q\mathcal{H}u(y)Q^T.\]
So \(\lambda_v(x) = \lambda_u(y)\), proving the case \(p=\infty\), since
\[\lambda_v(x) = \max_{|z|=1}z^T\mathcal{H}v(x)z = \max_{|z|=1}z^TQ\mathcal{H}u(y)Q^Tz = \max_{|z|=1}z^T\mathcal{H}u(y)z = \lambda_u(y).\]
Also
\[\Delta v(x) = \tr\mathcal{H}v(x) = \tr Q\mathcal{H}u(y)Q^T = \tr \mathcal{H}u(y) = \Delta u(y).\]
Therefore
\begin{align*}
(\mathcal{D}_p[u\circ T])(x) &= (\mathcal{D}_p v)(x)\\
                 &= (p-2)\lambda_v(x) + \Delta v(x)\\
								 &= (p-2)\lambda_u(y) + \Delta u(y)\\
								 &= (\mathcal{D}_pu)(T(x)).
\end{align*}
\end{proof}

\section{The proof of Theorem \ref{mainthm1}}

That dominative \(p\)-superharmonicity of the terms is a sufficient condition for the sum to be \(p\)-superharmonic is now an immediate consequence of Proposition \ref{D_p prop}:

\begin{proof}[Proof of Theorem \ref{mainthm1} (i)]
Let \(2\leq p\leq\infty\) and let \(u_1,\dots,u_N\) be dominative \(p\)-superharmonic \(C^2\) functions. That is, \(\mathcal{D}_pu_i\leq 0\). Write
\[V(x) := \sum_{i=1}^N u_i(x)\]
to denote the sum.
Then \(V\) is \(C^2\) and
\begin{align*}
\Delta_p V &\leq |\nabla V|^{p-2}\mathcal{D}_p\left[\sum_{i=1}^N u_i\right],\qquad\text{by the \emph{Domination} (1),}\\
           &\leq |\nabla V|^{p-2}\sum_{i=1}^N \mathcal{D}_p u_i,\qquad\text{by the \emph{Sublinearity} (2),}\\
					 &\leq 0.
\end{align*}
The calculations are the same when \(p = \infty\). 
\end{proof}
Notice that a sum of fundamental solutions, \(V(x) = \sum_{i=1}^N c_iw_{n,p}(x-~y_i)\) in a domain not containing the singularities, is just a special case by the \emph{Cylindrical Equivalence} (3) and the \emph{Invariance} (5).

\bigskip
We restate and prove the first part of Theorem \ref{mainthm1} (ii).

\begin{proposition}\label{lin}
Let \(2\leq p\leq\infty\) and let \(u\in C^2(\Omega)\). Then the following properties are equivalent.
\begin{enumerate}[a)]
	\item \(u(x) + a^Tx\) is \(p\)-superharmonic for every linear function \(a^Tx\).
	\item \(u(x) + cw_{n,p}(x-y)\) is \(p\)-superharmonic in \(\Omega\setminus\{y\}\) for every \(c\geq 0\) and every \(y\in\mathbb{R}^n\).
	\item \(u + u\circ T\) is \(p\)-superharmonic in \(\Omega\cap T^{-1}(\Omega)\) for every isometry \(T\colon\Omega\to\mathbb{R}^n\).
	\item \(u\) is a dominative \(p\)-superharmonic function.
\end{enumerate}
\end{proposition}

\begin{proof}

\begin{displaymath}
    \xymatrix{a)\ar@{<=>}[dr]&b)\ar@{<=>}[d]&c)\ar@{<=>}[dl]\\
					&d)&}
\end{displaymath}

The upward implications are immediate from the fundamental properties of \(\mathcal{D}_p\). As for the downward implications,
assume that \(u\) is \emph{not} dominative \(p\)-superharmonic. Then there is a point \(x_0\in\Omega\) so that \(\mathcal{D}_pu(x_0)>0\).
\bigskip

\noindent
\underline{a) \(\Rightarrow\) d):}
The implication is proved if we can find a constant \(a\in\mathbb{R}^n\) so that \(\Delta_p[u+a^Tx]>0\) at \(x_0\).

Let \(\xi_u = \xi_u(x_0)\) be a unit eigenvector of \(\mathcal{H}u(x_0)\) corresponding to the largest eigenvalue \(\lambda_u\) and let \(v(x) := a^Tx\) be the linear function with
\[a := \xi_u - \nabla u^T(x_0).\]
Then, at \(x_0\), \[\nabla u + \nabla v = \nabla u + a^T = \xi_u^T,\] so \(|\nabla(u+v)|=1\) and
\begin{align*}
\Delta_\infty[u+v] &= \od_\infty[u+v]\\
                   &= (\nabla u + \nabla v)(\mathcal{H}u + \mathcal{H}v)(\nabla u + \nabla v)^T\\
									 &= \xi_u^T(\mathcal{H}u + 0)\xi_u\\
									 &= \lambda_u\\
									 &= \mathcal{D}_\infty u(x_0) > 0
\end{align*}
if \(p=\infty\). Also, if \(p<\infty\),
\begin{align*}
\Delta_p[u + v] &= \od_p[u + v]\\
								&= (p-2)\od_\infty[u+v] + \Delta[u+v]\\
                &= (p-2)\lambda_u + \Delta u + 0\\
						    &= \mathcal{D}_p u(x_0) > 0.
\end{align*}

\bigskip

\noindent
\underline{b) \(\Rightarrow\) d):}
The implication is proved if we can find a \(y\in\mathbb{R}^n\) and a \(c\geq 0\) so that \[\Delta_p\big[u(x)+cw_{n,p}(x-y)\big]>0\] at \(x=x_0\).

Let \(\xi_u = \xi_u(x_0)\) be a unit eigenvector of \(\mathcal{H}u(x_0)\) corresponding to the largest eigenvalue \(\lambda_u\), and denote by \(q := \nabla u^T(x_0)\) the gradient of \(u\) at \(x_0\).
The idea is to consider a fundamental solution with centre far away from \(x_0\) in the proper direction, and then scale it in order to achieve a convenient cancellation in the sum of the gradients.

Introduce a (large) parameter \(s\) and let the centre of the scaled fundamental solution \(f_s(x) := c_sw_{n,p}(x-y_s)\) be at \(y_s := x_0 - q + s\xi_u\). Let \(z_s\) be the point
\[z_s := x_0 - y_s = q - s\xi_u.\]
Then
\[\nabla f_s(x_0) = c_s \nabla w_{n,p}(z_s) = c_s W_{n,p}'(|z_s|)\frac{z_s^T}{|z_s|},\qquad W_{n,p}(|x|) := w_{n,p}(x),\]
which equals \(-z_s^T = - (q - s\xi_u)^T\) if we choose the scale \(c_s\) to be
\[c_s := -\frac{|z_s|}{W_{n,p}'(|z_s|)} =
\begin{cases}
|q - s\xi_u|^\frac{n+p-2}{p-1},\qquad & 2\leq p<\infty,\\
|q - s\xi_u|, & p=\infty.
\end{cases}
\]
We may read the fraction \(\frac{n+p-2}{p-1}\) as \(1\) if \(p=\infty\).

We now get, at \(x_0\), \(\nabla u + \nabla f_s = q^T - (q - s\xi_u)^T = s\xi_u^T\) and
\begin{align*}
\od_p[u + f_s] &= (p-2)\frac{(\nabla u + \nabla f_s)(\mathcal{H}u + \mathcal{H}f_s)(\nabla u + \nabla f_s)^T}{|\nabla u + \nabla f_s|^2} + \Delta u+\Delta f_s\\
               &= (p-2)\xi_u^T(\mathcal{H}u + \mathcal{H}f_s)\xi_u + \Delta u+\Delta f_s\\
							 &= \mathcal{D}_pu + (p-2)\xi_u^T\mathcal{H}f_s\xi_u + \Delta f_s
\end{align*}
if \(p<\infty\) and
\[\od_\infty [u + f_s] = \mathcal{D}_\infty u + \xi_u^T\mathcal{H}f_s\xi_u\]
if \(p=\infty\).
Since \(\mathcal{D}_pu(x_0) > 0\) and does not depend on \(s\), we finish the proof by making the remaining term(s) arbitrarily close to zero.

The Hessian matrix of the fundamental solution is
\[\mathcal{H}w_{n,p}(x) = -\frac{W_{n,p}'(|x|)}{|x|}\left(\frac{n+p-2}{p-1}\frac{xx^T}{|x|^2} - I\right),\qquad \tfrac{n+\infty-2}{\infty-1} := 1,\]
so at \(z_s\) we get
\begin{align*}
\mathcal{H}f_s(x_0) &= c_s\mathcal{H}w_{n,p}(z_s)\\
                    &= -c_s\frac{W_{n,p}'(|z_s|)}{|z_s|}\left(\frac{n+p-2}{p-1}\frac{z_sz_s^T}{|z_s|^2} - I\right)\\
										&= \frac{n+p-2}{p-1}\frac{z_sz_s^T}{|z_s|^2} - I,\quad\text{and}\\
\Delta f_s(x_0) &= \frac{n+p-2}{p-1} - n.
\end{align*}
Thus, when \(p<\infty\),
\begin{align*}
(p-2)\xi_u^T\mathcal{H}f_s\xi_u + \Delta f_s
	&= (p-2)\left(\frac{n+p-2}{p-1}\frac{(z_s^T\xi_u)^2}{|z_s|^2} - 1\right) + \frac{n+p-2}{p-1} - n\\
	&= (p-2)\frac{n+p-2}{p-1}\left(\frac{(z_s^T\xi_u)^2}{|z_s|^2} - 1\right)\\
	&\to 0
\end{align*}
as \(s\to\infty\) since
\[\lim_{s\to\infty}\frac{(z_s^T\xi_u)^2}{|z_s|^2} = \lim_{s\to\infty}\frac{(q^T\xi_u - s)^2}{|q - s\xi_u|^2} = 1.\]
Likewise, when \(p=\infty\),
\[\xi_u^T\mathcal{H}f_s\xi_u = \frac{(z_s^T\xi_u)^2}{|z_s|^2} - 1 \to 0.\] 
To summarize. If \(\mathcal{D}_p u(x_0)>0\) and if \(s\) is large enough, then, for
\[z_s := \nabla u^T(x_0) - s\xi_u(x_0),\qquad f_s(x) := -\frac{|z_s|}{W_{n,p}'(|z_s|)}w_{n,p}(x - x_0 + z_s)\]
we have\footnote{We sometimes write \(\alpha_p\) as an abbreviation for \(p-2\) if \(p<\infty\) and 2 if \(p=\infty\).}
\[\Delta_p[u + f_s](x_0) = s^{\alpha_p}\od_p[u + f_s](x_0) > 0\]
and the sum \(u+f_s\) is not \(p\)-superharmonic.

\bigskip

\noindent
\underline{c) \(\Rightarrow\) d):}
Without loss of generality we may assume \(x_0\) to be the origin. i.e. \(0\in\Omega\) and \(\mathcal{D}_pu(0)>0\). We shall prove the implication by finding an isometry \(T\) and a point \(y_0\in\Omega\), equal or close to 0, so that \(T(y_0)\in\Omega\) and \(\Delta_p \big[u + u\circ T\big] > 0\) at \(y_0\).

Let \(y\in\mathbb{R}^n\), \(|y| = 1\) be a fixed direction in space defining the line
\[\ell := \{\alpha y\;|\; \alpha\in\mathbb{R}\}.\]
The \emph{projection} onto \(\ell\) is given by the 1-rank matrix
\[P := yy^T.\]
We have, as for every projection,
\[Px\in\ell\qquad\text{and}\qquad PP = P.\]
The \emph{reflection} about \(\ell\) is now given by \(Rx := Px - (x-Px)\). That is,
\[R = 2P - I.\]
A reflection satisfies
\[R|_\ell = \mathrm{id}\qquad\text{and}\qquad RR = I.\]
After carefully choosing \(y\), \(T(x) := Rx\) will be our isometry.

Define the superposition
\[V(x) := \frac{u(x) + u(Rx)}{2}.\]
The main idea of the proof is that, on \(\ell\), \(\nabla V\) will be pointing in the \(y\)-direction:

The chain rule gives \(\nabla V(x) = \tfrac{1}{2}\big(\nabla u(x) + \nabla u(Rx)R\big)\) and \(\mathcal{H}V(x) = \tfrac{1}{2}\big(\mathcal{H}u(x) + R\mathcal{H}u(Rx)R\big)\) and \(\Delta V(x) = \tfrac{1}{2}\big(\Delta u(x) + \Delta u(Rx)\big)\).
For \(x\in\ell\) we have \(Rx=x\), and
\begin{align*}
\nabla V &= \nabla u\frac{I + R}{2} = \nabla uP \in\ell,\\
\mathcal{H}V &= \frac{\mathcal{H}u + R\mathcal{H}uR}{2},\\
\Delta V &= \Delta u.
\end{align*}
This gives, when \(\nabla V = \nabla u P\neq 0\),
\begin{align*}
\od_pV\Big|_\ell &= (p-2)\frac{\nabla V\mathcal{H}V\nabla V^T}{|\nabla V|^2} + \Delta V\\
             &= (p-2)\frac{1}{|\nabla uP|^2}\nabla u P\frac{\mathcal{H}u + R\mathcal{H}uR}{2}P\nabla u^T + \Delta u\\
						 &= (p-2)\frac{\nabla u P\mathcal{H}uP\nabla u^T}{|\nabla uP|^2} + \Delta u\\
						 &= (p-2)y^T\mathcal{H}uy + \Delta u
\end{align*}
since \(RP = P\) and \(0\neq P\nabla u^T\) is parallel to \(y\).
Similarly
\[\od_\infty V\Big|_\ell = y^T\mathcal{H}uy.\]
Now choose \(y := \xi_u\) where \(\xi_u = \xi_u(0)\) is a unit eigenvector of \(\mathcal{H}u(0)\) corresponding to the largest eigenvalue \(\lambda_u\). The isometry \(T\) is then
\[T(x) := Rx = (2P-I)x = (2\xi_u\xi_u^T - I)x.\]

Since \(0\in\ell\), and unless \(\nabla u(0)\xi_u = 0\), it follows that
\begin{align*}
\frac{1}{2}\od_p [u + u\circ T]\Big|_{x=0} &= \od_pV(0)\\
                                  				 &= \mathcal{D}_p u(0) > 0
\end{align*}
and \(\Delta_p[u + u\circ T] = 2^{\alpha_p+1}|\nabla uP|^{\alpha_p}\od_p[u + u\circ T] > 0\) at \(x=0\) since \(\nabla u(0)P = \nabla u(0)\xi_u\xi_u^T \neq 0\).

If \(\nabla u(0)\xi_u = 0\) we complete the proof with a continuity argument:
Since \(\lambda_u(0) > 0\),\footnote{or else
\(\mathcal{D}_p u = \lambda_1 + \cdots + \lambda_{n-1} + (p-1)\lambda_u \leq 0\) at \(x=0\).}
 \(\mathcal{D}_p u(0)>0\) and \(\Omega\) is open, and since the Hessian is continuous, there must be a common \(\epsilon > 0\) so that
\begin{align*}
\xi_u^T\mathcal{H}u(t\xi_u)\xi_u &> 0,\qquad\text{and}\\
(p-2)\xi_u^T\mathcal{H}u(t\xi_u)\xi_u + \Delta u(t\xi_u)&> 0,\qquad\text{if \(p<\infty\), and}\\
T(t\xi_u) = t\xi_u &\in\Omega,\qquad\forall\; t\in[0,\epsilon].
\end{align*}
A Taylor expansion of the gradient about 0 in the \(\xi_u\)-direction then gives
\[\nabla u(\epsilon\xi_u) = \nabla u(0) + \epsilon\xi_u^T\mathcal{H}u(t_0\xi_u)\]
for some \(t_0\in [0,\epsilon]\). So
\[\nabla u(\epsilon\xi_u)\xi_u = 0 + \epsilon\xi_u^T\mathcal{H}u(t_0\xi_u)\xi_u > 0,\]
and again, since \(\epsilon\xi_u\in\ell\),
\begin{align*}
\frac{1}{2}\Delta_p [u + u\circ T]\Big|_{x=\epsilon\xi_u}
	&= |2\nabla u(\epsilon\xi_u)P|^{p-2}\Big((p-2)y^T\mathcal{H}u(\epsilon\xi_u)y + \Delta u(\epsilon\xi_u)\Big)\\
	&= \Big(2\nabla u(\epsilon\xi_u)\xi_u\Big)^{p-2}\Big((p-2)\xi_u^T\mathcal{H}u(\epsilon\xi_u)\xi_u + \Delta u(\epsilon\xi_u)\Big)\\
	&>0
\end{align*}
if \(p<\infty\) and
\[\frac{1}{2}\Delta_\infty [u + u\circ T]\Big|_{x=\epsilon\xi_u} = \Big(2\nabla u(\epsilon\xi_u)\xi_u\Big)^2\xi_u^T\mathcal{H}u(\epsilon\xi_u)\xi_u > 0\]
if \(p=\infty\).
Thus the sum \(u + u\circ T\) is not \(p\)-superharmonic in \(\Omega\cap T^{-1}(\Omega)\).
\end{proof}

We finish the proof of Theorem \ref{mainthm1} by showing the equivalence of (d) and (e). The nontrivial implication is (d)\(\Rightarrow\)(e). Namely that if \(2<p<\infty\) and \(u\in C^2(\Omega)\) is both \(p\)-harmonic and dominative \(p\)-superharmonic, then \(u\) is a cylindrical fundamental solution. Since the hypothesis and the domination implies
\[0 = \Delta_p u \leq |\nabla u|^{p-2}\mathcal{D}_p u \leq 0,\]
the claim follows from Proposition \ref{segre} below. It is partially the converse of the Cylindrical Equivalence.

\begin{proposition}\label{segre}
Let \(2<p<\infty\) and let \(u\in C^2(\Omega)\). If

\begin{equation}
\Delta_p u = 0 = \mathcal{D}_p u\qquad\text{in \(\Omega\)},
\label{double}
\end{equation}
then \(u\) is locally a cylindrical fundamental solution.
\end{proposition}

The proof relies on a  rather deep result in differential geometry. We refer to \cite{Cecil2015}, \cite{MR901710} and \cite{Thorbergsson2000963} for the details of the following exposition.

A nonconstant smooth function \(u\colon M\to\mathbb{R}\) on a Riemannian manifold \(M\) is called \emph{isoparametric} if there exists functions \(f\) and \(g\) so that
\begin{equation}
\frac{1}{2}|\nabla u|^2 = f(u)\qquad\text{and}\qquad \Delta u = g(u).
\label{isomap}
\end{equation}
A regular level-set of an isoparametric function is called an \emph{isoparametric hypersurface}.

The isoparametric hypersurfaces in the Euclidean case \(M\subseteq\mathbb{R}^n\) have been completely classified. Apparently, this was first done by Segre (see \cite{Thorbergsson2000963}) in 1938:

\begin{theorem}[Segre]
A connected isoparametric hypersurface in \(\mathbb{R}^n\) is, upon scaling and an Euclidean motion, an open part of one of the following hypersurfaces:
\begin{enumerate}
	\item a hyperplane \(\mathbb{R}^{n-1}\),
	\item a sphere \(S^{n-1}\),
	\item a generalized cylinder \(S^{k-1}\times\mathbb{R}^{n-k}\), \(k = 2,\dots,n-1\).
\end{enumerate}
\end{theorem}
Moreover, the family of cylinders \(u^{-1}(c)\) is concentric. Thus \(u\) is a function of the distance to the common ``axis'' of the cylinders, the axis being a \((n-k)\)-dimensional affine subspace, \(k = 1,\dots,n\), in \(\mathbb{R}^n\). Call this subset \(\mathcal{A}_k\).

The axis \(\mathcal{A}_k\) is isomorphic to \(\mathbb{R}^{n-k}\),
\[\mathcal{A}_k \cong \mathbb{R}^{n-k} \cong \left\{
\begin{pmatrix}
	0 & 0\\ 0 & I_{n-k}
\end{pmatrix}x\; \middle|\; x\in\mathbb{R}^n\right\} =: \tilde{\mathcal{A}_k},\]
where, obviously,
\[\operatorname{dist}(x,\tilde{\mathcal{A}_k}) = \sqrt{x_1^2 + \cdots + x_k^2} = |(I_k\; 0)x|,\qquad 0\in\mathbb{R}^{k\times n-k}.\]
Translating and rotating back to \(\mathcal{A}_k\) via an isometry \(Q_n^T(x-x_0)\), \(Q_n\) \(n\times n\) orthogonal, we find that
\begin{align*}
u(x) &= U(\operatorname{dist}(x,\mathcal{A}_k))\\
     &= U\left(|(I_k\; 0)Q^T_n(x-x_0)|\right)\\
		 &= U\left(|Q^T_k(x-x_0)|\right)
\end{align*}
where \(Q_k\) is the \(n\times k\) matrix consisting of the first \(k\) columns of \(Q_n\).

Thus an isoparametric function is cylindrical.

\begin{proof}[Proof of Proposition \ref{segre}]
If \(u\) is affine, it can locally be written as a 1-cylindrical fundamental solution.
If \(u\) is not an affine function, let
 \(x_0\in\Omega\) be a point with a neighbourhood \(\Omega'\subseteq\Omega\) where \(u\) has connected level-sets and where \(\nabla u\neq 0\), \(\Hu\neq 0\).
By the discussion above and the last part of Proposition \ref{D_p prop} (3), it is sufficient to show that \(u\) is a smooth isoparametric function.

As \(\nabla u\neq 0\) and \(p>2\), our equation \eqref{double}
\[|\nabla u|^{p-2}\left((p-2)\hu\Hu\hu^T + \Delta u\right) = 0 = (p-2)\lambda_u + \Delta u\]
implies \(\hu\Hu\hu^T = \lambda_u\) and the gradient is therefore an eigenvector of the Hessian:
\[\Hu\nabla u^T = \lambda_u\nabla u^T.\]
Let \(\bfc\) be a differentiable curve in a level-set of \(u\). Then \(0 = \frac{\dd}{\dd t}u(\bfc(t)) = \nabla u\frac{\dd\bfc}{\dd t}\) and
\[\frac{\dd}{\dd t}\frac{1}{2}|\nabla u(\bfc(t))|^2 = \nabla u\Hu\frac{\dd\bfc}{\dd t} = \lambda_u\nabla u\frac{\dd\bfc}{\dd t} = 0.\]
Thus the length of the gradient is constant on the level-sets and can be written as a function only of \(u\). Say,
\begin{equation}
\frac{1}{2}|\nabla u(x)|^2 = f(u(x))>0.
\label{gradfun}
\end{equation}

We need to show that \(f\) is differentiable, because if so, differentiation of \eqref{gradfun} yields
\begin{equation}
\nabla u\Hu = f'(u)\nabla u
\label{lamisf}
\end{equation}
and \(\lambda_u = f'(u)\).
Using \eqref{double} once more, we then find that also \(\Delta u\) is a function of \(u\):
\begin{equation}
\Delta u =  - (p-2)\lambda_u =  - (p-2)f'(u) =: g(u).
\label{Lapfun}
\end{equation}

Fix \(x\in \Omega'\) and let \(\bfx(t)\) now be an integral curve of the gradient field starting from \(x\):
\[\frac{\dd\bfx}{\dd t}(t) = \nabla u^T(\bfx(t)),\qquad \bfx(0) = x.\]
Then define the function \(h\) as \(h(t) := u(\bfx(t))\). We see that \(h\) is \(C^2\) and
\[h'(t) = \nabla u(\bfx(t))\nabla u^T(\bfx(t)) = 2f(u(\bfx(t))) = 2f(h(t)) >0.\]
Thus \(h\) is strictly monotone and
\begin{align*}
\frac{1}{2}\frac{h''(0)}{h'(0)} &= \frac{1}{2}\lim_{t\to 0}\frac{h'(t) - h'(0)}{h(t) - h(0)}\\
																&= \lim_{t\to 0}\frac{f(h(t)) - f(h(0))}{h(t) - h(0)}\\
																&= f'(h(0)) = f'(u(x)).
\end{align*}
This is enough to conclude that \eqref{lamisf}, and thus \eqref{Lapfun}, is valid.

As for the regularity of \(u\), observe that
if \(F\) is an anti-derivative of \((2f)^\frac{p-2}{2}\) and \(\phi(x) := F(u(x))\), then \(\phi\) is \(C^2\) and
\[\nabla\phi = F'(u)\nabla u = |\nabla u|^{p-2}\nabla u.\]
That is, \(\phi\) is harmonic by \eqref{double}. It follows that \(\phi\) is real-analytic, and so is \(u\) since \(\nabla u = |\nabla\phi|^{-\frac{p-2}{p-1}}\nabla\phi\).
\end{proof}

This concludes the proof of Theorem \ref{mainthm1}.

It is worth noting that Proposition \ref{segre} is \emph{not} true for \(p=2\) and \(p=\infty\). The case \(p=2\) is obvious since \(\mathcal{D}_2 \equiv \Delta\) and every harmonic function satisfies \eqref{double}.

When \(p=\infty\), we believe a counter example is provided by a function
\[u(x) = \operatorname{dist}(x,\partial\Omega),\qquad \text{\(x\) near \(\partial\Omega\)},\]
where \(\Omega\) is a smooth, bounded and strictly convex, but not spherical, domain in, say, \(\mathbb{R}^2\).
Then \(u\) is neither affine nor a circular cone (i.e. not a cylindrical \(\infty\)-fundamental solution). But \(u\) solves the
\emph{eikonal equation}
\[|\nabla u| = 1\quad \text{in \(\Omega\)},\qquad u = 0\quad \text{on \(\partial\Omega\)}\]
so \(\nabla u\Hu = 0 = 0\cdot\nabla u\) and
\[\mathcal{D}_\infty u = \lambda_u = 0 = \nabla u\Hu\nabla u^T = \Delta_\infty u\]
since \(\lambda_u = 0\) is the larger eigenvalue as \(u\) surely has regions where it is locally concave.

\section{Viscosity solutions}\label{sect_vis}
The equation \(\D_pu = 0\) needs to be interpreted in the \emph{viscosity sense} (v.s.). We refer to \cite{MR1118699}, \cite{MR2084272} and \cite{MR3289084} for the general theory of viscosity solutions. For our purpose, only the basic notions of the concept are needed.

A PDE \(F(\nabla u,\mathcal{H}u) = 0\) is said to be \textbf{degenerate elliptic} if for any two symmetric matrices \(X\) and \(Y\) such that \(Y-X\) is positive semi-definite, i.e. \(X\leq Y\), we have
\[F(q,X) \leq F(q,Y)\]
for all \(q\in\mathbb{R}^n\).\footnote{The definition is often made with the opposite inequality and \(F\) replaced with \(-F\). The sign is not essential and we stick to the convention used in our reference \cite{MR3289084}.}

In our case
\[0 = \mathcal{D}_p u = F_p(\mathcal{H}u)\]
where
\[F_p(X) :=
\begin{cases}
 (p-2)\lambda_X + \tr X,\qquad &p<\infty,\\
 \lambda_X, &p=\infty.
\end{cases}\]
So if \(z^TXz\leq z^TYz\) for all \(z\in\mathbb{R}^n\), then
\begin{align*}
\lambda_X &= \xi_X^T X\xi_X\\
             &\leq \xi_X^TY\xi_X\\
						 &\leq \lambda_Y
\end{align*}
and \(F_\infty(X) \leq F_\infty(Y)\). Also, for any orthonormal set \(\{z_1,\dots,z_n\}\),
\[\tr X = \sum_{i=1}^n z_i^TX z_i \leq \sum_{i=1}^n z_i^TY z_i = \tr Y,\]
so \(F_p(X) \leq F_p(Y)\)
when \(0\leq p-2<\infty\). Thus the dominative \(p\)-Laplace equation \(\mathcal{D}_pu = 0\) is degenerate elliptic.

It is known that the \(p\)-Laplace Equation is degenerate elliptic for \(2\leq p\leq\infty\).

\subsection{Definitions and fundamental properties}

Consider a degenerate elliptic equation
\begin{equation}
F(\nabla u,\mathcal{H}u) = 0.
\label{pde}
\end{equation}

\begin{definition}\label{visdef}
We say that \(u\colon\Omega\to(-\infty,\infty]\) is a \textbf{viscosity supersolution} of the PDE \eqref{pde} if
\begin{enumerate}
	\item \(u\) is lower semi-continuous (l.s.c)
	\item \(u<\infty\) in a dense subset of \(\Omega\)
	\item If \(x_0\in\Omega\) and \(\phi\in C^2\) touches \(u\) from below at \(x_0\), i.e.
\[\phi(x_0) = u(x_0),\qquad \phi(x)\leq u(x)\qquad\text{for \(x\) near \(x_0\)},\]
we require that
\[F(\nabla\phi(x_0),\mathcal{H}\phi(x_0)) \leq 0.\]	
\end{enumerate}
\end{definition}

The viscosity \emph{sub}solutions \(u\colon\Omega\to[-\infty,\infty)\) are defined in a similar way: they are upper semicontinuous and the test functions touch from above. Finally, a function \(u\colon\Omega\to\mathbb{R}\) is a \textbf{viscosity solution} if it is both a viscosity supersolution and a viscosity subsolution. Necessarily, \(u\in C(\Omega)\).

We shall say that \(u\) is \emph{dominative \(p\)-superharmonic} if \(\D_pu\leq 0\) v.s.
The \(p\)-superharmonic functions were traditionally defined by the \emph{comparison principle} and weak integral formulations -- and not by viscosity. According to \cite{MR1871417}, however, the two concepts are equivalent and we may therefore define the \(p\)-superharmonic functions in terms of viscosity as well. The comparison principle then becomes a theorem:

\begin{theorem}[Comparison Principle]\label{psupdef}
Let \(2\leq p\leq\infty\). Assume that \(v\) is \(p\)-subharmonic and that \(u\) is \(p\)-superharmonic in \(\Omega\). Let \(D\subset\subset\Omega\). Then
\[v|_{\partial D}\leq u|_{\partial D}\qquad\Rightarrow\qquad v\leq u\qquad\text{in \(D\)}.\]
\end{theorem}

Before we extend the fundamental properties of the dominative operator (Proposition \ref{D_p prop}) to the setting of viscosity, we establish that a dominative \(\infty\)-superharmonic function is the same as a concave function:
\begin{proposition}\label{conlem}
Let \(\Omega\subseteq\mathbb{R}^n\) be open and convex. Then
\[\mathcal{D}_\infty u\leq 0\;\text{v.s in \(\Omega\)}\qquad\iff\qquad \text{\(u\) is concave in \(\Omega\)}.\]
\end{proposition}

This is Proposition 4.1 in \cite{Lindqvist2000} or, alternatively, Theorem 2.2 in \cite{MR2817413} in disguise.
Note that continuity is automatically given by either direction: It is well known that concave functions are continuous in open domains. Also, if \(\mathcal{D}_\infty u\leq 0\) v.s. then \(u\) is \(\infty\)-superharmonic by Proposition \ref{D_p prop vis} below and is therefore continuous by Lemma 6.7 found in \cite{MR3467690}.

\begin{proposition}[Fundamental properties of \(\mathcal{D}_p\). Viscosity sense]\label{D_p prop vis}
The following hold for \(2\leq p\leq\infty\).
\begin{enumerate}[(1)]
	\item \textbf{Domination:}
	\[\text{\(u\) is dominative \(p\)-superharmonic}\qquad\Rightarrow\qquad \text{\(u\) is \(p\)-superharmonic.}\]
	\item \textbf{Sublinearity:} If \(\mathcal{D}_p u\leq 0\) and \(\mathcal{D}_p v\leq 0\) v.s. and \(u\) is radial, then \(\mathcal{D}_p[u+v]\leq 0\) v.s.
	\item \textbf{Radial equivalence:} If \(u\) is a radial \(p\)-superharmonic function, then \(\mathcal{D}_p u\leq 0\) v.s.
	\item \textbf{Nesting property:}
	\begin{itemize}
		\item If \(\mathcal{D}_p u\leq 0\) v.s., then \(\mathcal{D}_q u\leq 0\) v.s. for every \(2\leq q\leq p\).
		
		In particular, if \(u\) is locally concave, then \(\mathcal{D}_qu\leq 0\) v.s. for all \(2\leq q\leq\infty\).
		\item If \(\mathcal{D}_p u\geq 0\) v.s., then \(\mathcal{D}_q u\geq 0\) v.s. for every \(p\leq q\leq \infty\).
		
		In particular, if \(u\) is subharmonic then \(\mathcal{D}_qu\geq 0\) v.s. for all \(2\leq q\leq\infty\).
	\end{itemize}
	\item \textbf{Invariance:} If \(\mathcal{D}_p u \leq 0\) v.s. in \(\mathbb{R}^n\), then \(\mathcal{D}_p[u\circ T]\leq 0\) v.s. in \(\mathbb{R}^n\) for all isometries \(T\colon \mathbb{R}^n\to\mathbb{R}^n\).
\end{enumerate}
\end{proposition}

The claims (1), (4) and (5) follow immediately from the corresponding properties in the smooth case (Proposition \ref{D_p prop}). The proofs of (2) and (3) are more difficult and are postponed until the survey of \emph{Radial \(p\)-superharmonic functions} in Section \ref{radpsup}.

\begin{proof}\textbf{Domination:}\\
Assume \(\mathcal{D}_p u \leq 0\) v.s. Let \(x_0\) be a point in the domain of \(u\), and assume \(\phi\) is a test function of \(u\) from below at \(x_0\). Then \(\mathcal{D}_p\phi(x_0)\leq 0\) and
\[\Delta_p\phi(x_0) \leq
\begin{cases}
|\nabla\phi(x_0)|^{p-2}\mathcal{D}_p\phi(x_0) \leq 0,\qquad &2\leq p<\infty\\
|\nabla\phi(x_0)|^2 \mathcal{D}_\infty\phi(x_0)\leq 0,\qquad &p = \infty,
\end{cases}
\]
by the smooth case \emph{Domination}. Hence \(\Delta_p u\leq 0\) v.s. and \(u\) is \(p\)-superharmonic.
\end{proof}

\begin{proof}\textbf{Nesting Property:}\\
Let \(2\leq q\leq p\) and suppose \(\mathcal{D}_p u\leq 0\) v.s. Let \(x_0\) be a point in the domain of \(u\), and assume \(\phi\) is a test function of \(u\) from below at \(x_0\). Then \(\mathcal{D}_p\phi(x_0) \leq 0\) and
\[\mathcal{D}_q\phi(x_0) \leq 0\]
by the smooth case \emph{Nesting Property} and \(\mathcal{D}_q u\leq 0\) v.s.

The additional claim follows from Proposition \ref{conlem}.

Let \(p\leq q\leq\infty\) and suppose \(\mathcal{D}_p u\geq 0\) v.s. Let \(x_0\) be a point in the domain of \(u\), and assume \(\phi\) is a test function of \(u\) from above at \(x_0\). Then \(\mathcal{D}_p\phi(x_0) \geq 0\) and
\[\mathcal{D}_q\phi(x_0) \geq 0\]
by the smooth case \emph{Nesting Property} and \(\mathcal{D}_q u\geq 0\) v.s.
\end{proof}

\begin{proof}\textbf{Invariance:}\\
Assume \(\mathcal{D}_p u \leq 0\) v.s. in \(\mathbb{R}^n\) and let \(T\colon \mathbb{R}^n\to\mathbb{R}^n\) be an isometry.
Let \(x_0\in\mathbb{R}^n\), and assume \(\phi\) is a test function for \(u\circ T\) from below at \(x_0\).
Then the function \(\hat{\phi} := \phi\circ T^{-1}\) is a test function for \(u\) from below at \(y_0 := T(x_0)\), since
\[\hat{\phi}(y_0) = \phi\circ T^{-1}(T(x_0)) = \phi(x_0) = u(y_0)\]
and, for \(y\) near \(y_0\), \(T^{-1}(y)\) is near \(x_0\) and
\[\hat{\phi}(y) = \phi(T^{-1}(y)) \leq (u\circ T)(T^{-1}(y)) = u(y).\]
Thus \(\mathcal{D}_p\hat{\phi}(y_0) \leq 0\) and
\begin{align*}
\mathcal{D}_p\phi(x_0) &= (\mathcal{D}_p[\hat{\phi}\circ T])(x_0)\\
             &= (\mathcal{D}_p\hat{\phi})(T(x_0)),\qquad\text{by the smooth case \emph{Invariance},}\\
						 &= \mathcal{D}_p\hat{\phi}(y_0)\\
						 &\leq 0
\end{align*}
and \(\mathcal{D}_p[u\circ T]\leq 0\) v.s. in \(\mathbb{R}^n\).
\end{proof}

\subsection{The superposition principle for radial p-superharmonic functions}

Proposition \ref{D_p prop vis} contains everything we need in order to show that radial \(p\)-superharmonic functions can be added:

\begin{proof}[Proof of Theorem \ref{mainthm2}]
Let \(2\leq p\leq\infty\) and let \(u_1,\dots,u_N\) be radial \(p\)-superharmonic functions in \(\mathbb{R}^n\). We must show that the sum
\[\sum_{i=1}^N u_i(x-y_i) + K(x),\qquad y_i\in\mathbb{R}^n,\]
is \(p\)-superharmonic in \(\mathbb{R}^n\) for any concave function \(K\).

Observe first that if \(\D_p u\leq 0\), \(\D_p v\leq 0\) v.s. where \(u\) is radial, then \(\D_p[v\circ T^{-1}]\leq 0\) v.s for every isometry \(T\) by the \emph{Invariance} (5).
By the \emph{Sublinearity} (2) it then follows that
\(\mathcal{D}_p[u + v\circ T^{-1}]\leq 0\) v.s. 
and
\[\mathcal{D}_p[u\circ T + v] = \mathcal{D}_p\big[(u + v\circ T^{-1})\circ T\big]\leq 0\; \text{ v.s.,}\]
again by (5).

From the \emph{Radial Equivalence (3)}, \(\mathcal{D}_pu_i\leq 0\) v.s for each \(i = 1,\dots,N\) and also \(\D_p K\leq 0\) v.s. by the \emph{Nesting Property} (4). Denoting the translations by \(T_i(x) := x-y_i\), and adding the functions \(u_i\circ T_i\) one by one, starting with \(K\), we obtain that
\[\mathcal{D}_p\left[\sum_{i=1}^N u_i\circ T_i + K\right]\leq 0\;\text{ v.s.}\]
We now use the \emph{Domination} (1) to conclude that the sum
\[\sum_{i=1}^N u_i\circ T_i + K\]
is \(p\)-superharmonic in \(\mathbb{R}^n\).
\end{proof}

\section{Radial p-superharmonic functions}\label{radpsup}

A radial \(p\)-superharmonic function \(u\colon\mathbb{R}^n\to (-\infty,\infty]\) is on the form \(u(x) = U(|x|)\) for some l.s.c. one-variable function \(U\colon [0,\infty)\to(-\infty,\infty]\). By comparing with constant functions (which are \(p\)-harmonic), it is clear that \(U\) must be decreasing (\(=\) non-increasing). Also, since the set \(\{x\colon u(x) = \infty\}\) has measure zero \cite{Lindqvist1986}, the origin is the only possible pole of \(u\). Therefore, \(u\) is \emph{bounded} in every annulus
\[A_a^b := B(0,b)\setminus\overline{B(0,a)},\qquad 0<a<b<\infty.\]
Equivalently, \(U\) is bounded on the interval \((a,b)\).

Let
\[w_{n,p}(x) = W_{n,p}(|x|),\qquad 2\leq p\leq\infty,\;n\geq 2,\]
denote the fundamental solution to the \(p\)-Laplace Equation in \(\mathbb{R}^n\), see \eqref{fundsol} in the Introduction. In this section, the important properties of the fundamental solution is that any scaled version, \(C_1w_{n,p} + C_2\), (\(C_1,C_2\in\mathbb{R}\)), is still a radial \(C^\infty\) \(p\)-harmonic function in \(\mathbb{R}^n\setminus\{0\}\). And when \(C_1\geq 0\), it is also dominative \(p\)-harmonic by Proposition \ref{D_p prop} part 3.

Furthermore, we shall frequently use the fact that \(W_{n,p} \colon [0,\infty)\to(-\infty,\infty]\) is \emph{strictly} decreasing. A simple calculation shows that \emph{a scaled fundamental solution is uniquely determined by its values at two different positive radii.}

Given a radial \(p\)-superharmonic function \(u(x) = U(|x|)\) in \(\mathbb{R}^n\) and two numbers \(0<a<b\), we define \(h_{ab}\) on \(\mathbb{R}^n\setminus\{0\}\) as the scaled fundamental solution \(h_{ab}(x) = H_{ab}(|x|)\) where
\begin{align}\label{hfundI}
H_{ab}(r) &:= C_{ab}\Big[W_{n,p}(r) - W_{n,p}(b)\Big] + U(b),\\
C_{ab} &:= \frac{U(a)-U(b)}{W_{n,p}(a)-W_{n,p}(b)}.\label{hfundII}
\end{align}
\begin{figure}[h]
\centering
\includegraphics{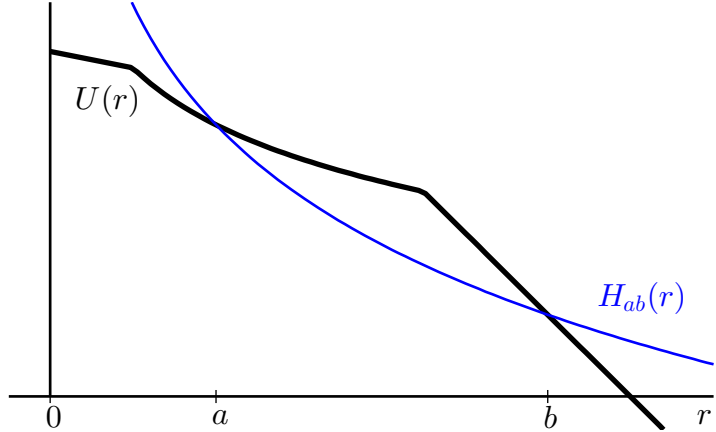}%
\caption{\textcolor{blue}{\(H_{ab}\)}\(\leq U\) on \([a,b]\), while \textcolor{blue}{\(H_{ab}\)}\(\geq U\) outside the interval.}%
\label{figradsup1}%
\end{figure}
The point of this is that \(h_{ab}\) is \(p\)-harmonic, smooth and it satisfies
\[H_{ab}(a) = U(a)\qquad\text{and}\qquad H_{ab}(b) = U(b).\]
Thus \(h_{ab}=u\) on the boundary of the annulus \(A_a^b\) and by the comparison principle we must have
\begin{equation}
h_{ab} \leq u\qquad\text{in } A_a^b.
\label{comp1}
\end{equation}
Equivalently,
\[H_{ab} \leq U\qquad\text{in } (a,b).\]

We now deduce other immediate properties of \(H_{ab}\) and the scaling constant \(C_{ab}\).

\begin{lemma}\label{hlem}
Let \(2\leq p\leq\infty\) and let \(u(x) = U(|x|)\) be a given radial \(p\)-superharmonic function in \(\mathbb{R}^n\).
For numbers \(0<a<b\) define the scaled fundamental solution \(h_{ab}(x) = H_{ab}(|x|)\) with scaling constant \(C_{ab}\) as in \eqref{hfundI} and \eqref{hfundII}.
\begin{enumerate}[(1)]
	\item We have the opposite inequality outside the annulus \(A_a^b\):
	\[H_{ab} \geq U\qquad\text{in } (0,a]\cup[b,\infty).\]
	\item 
	\(0 \leq C_{ab} \leq C_{bc} < \infty\)
	whenever \(0<a<b<c\).
	\item The mappings \(a\mapsto C_{ab}\) and \(c\mapsto C_{bc}\) are \emph{increasing.}
	\item The one sided limits
	\[C_b^- := \lim_{a\to b^-}C_{ab},\qquad C_b^+ := \lim_{c\to b^+}C_{bc}\]
	exist and
	\[0\leq C_b^-\leq C_b^+ < \infty.\]
\end{enumerate}
\end{lemma}
Observe that the existence of the limits (4) implies that \(U(r)\) has one sided derivatives at every \(r\neq 0\).
For example,
\[\frac{U(a) - U(b)}{a-b} = \frac{U(a) - U(b)}{W_{n,p}(a)-W_{n,p}(b)}\frac{W_{n,p}(a)-W_{n,p}(b)}{a-b}\]
which goes to \(C_b^- W_{n,p}'(b)\) as \(a\to b^-\).

\begin{figure}[h]
\centering
\includegraphics{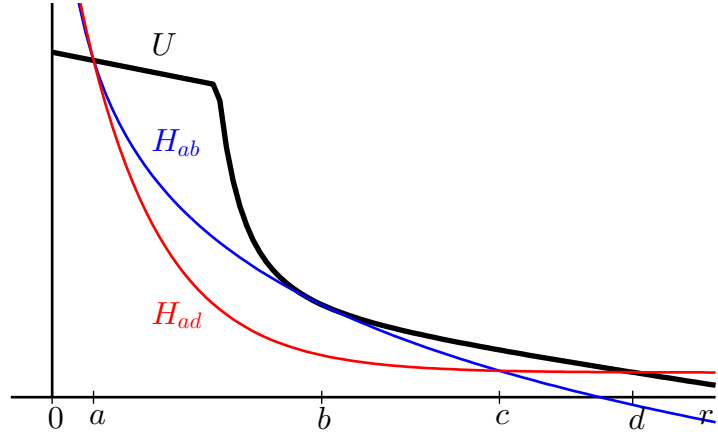}%
\caption{Impossible situation. If there is a \(d>b\) where \textcolor{blue}{\(H_{ab}(d)\)} \(< U(d)\), then \(\exists c\in[b,d)\) so that \textcolor{red}{\(H_{ad}(c)\)} \(=\)  \textcolor{blue}{\(H_{ab}(c)\)}. Thus \textcolor{red}{\(H_{ad}\)} \(\equiv\)  \textcolor{blue}{\(H_{ab}\)}.}%
\label{figradsup5}%
\end{figure}

\begin{proof}[Proof of (1)]
Suppose for the sake of contradiction that there is a number \(d>b\) so that \(H_{ab}(d) < U(d)\). See Figure \ref{figradsup5}.

The function \(h_{ad}(x) = H_{ad}(|x|)\) satisfies
\begin{align*}
H_{ad}(a) &= U(a) = H_{ab}(a),&&\text{by definition},\\
H_{ad}(d) &= U(d) > H_{ab}(d),&&\text{by assumption},\\
H_{ad}(b) &\leq U(b) = H_{ab}(b),&&\text{by the comparison principle.}
\end{align*}
By the Intermediate Value Theorem, there is an \(c\in[b,d)\) so that \(H_{ad}(c) = H_{ab}(c)\).
Since also \(H_{ad}(a) = H_{ab}(a)\), the two functions are identical. This is the contradiction
to the assumption \(H_{ad}(d) > H_{ab}(d)\). The proof for \(0<r<a\) is symmetric.
\end{proof}

\begin{figure}[h]
\centering
\includegraphics{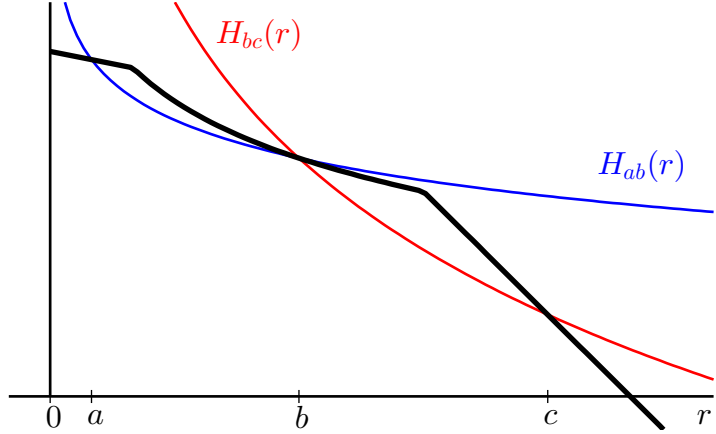}%
\caption{Proof of (2). \textcolor{red}{\(H_{bc}\)} \(\leq\) \textcolor{blue}{\(H_{ab}\)} on \((b,c)\) so \textcolor{red}{\(H_{bc}'(b)\)} \(\leq\) \textcolor{blue}{\(H_{ab}'(b)\)}.}%
\label{figradsup3}%
\end{figure}

\begin{proof}[Proof of (2)]
\[0\leq C_{ab} := \frac{U(a)-U(b)}{W_{n,p}(a)-W_{n,p}(b)}<\infty\]
for all \(0<a<b\) since \(U\) is decreasing and \(W_{n,p}\) is strictly decreasing. Moreover,
\[H_{bc}(b+\epsilon) \leq U(b+\epsilon) \leq H_{ab}(b+\epsilon)\]
whenever \(b+\epsilon\) is between \(b\) and \(c\), see Figure \ref{figradsup3}. The first inequality follows from the comparison principle, and the second inequality follows from (1). Therefore
\begin{align*}
C_{ab}W_{n,p}'(b) &= H_{ab}'(b)\\
                  &= \lim_{\epsilon\to0^+}\frac{H_{ab}(b+\epsilon) - H_{ab}(b)}{\epsilon}\\
									&\geq \lim_{\epsilon\to0^+}\frac{H_{bc}(b+\epsilon) - H_{bc}(b)}{\epsilon}\\
									&= H_{bc}'(b)\\
									&= C_{bc}W_{n,p}'(b)
\end{align*}
and \(C_{ab} \leq C_{bc}\) since \(W_{n,p}'(b)<0\).
\end{proof}

\begin{figure}[h]
\centering
\includegraphics{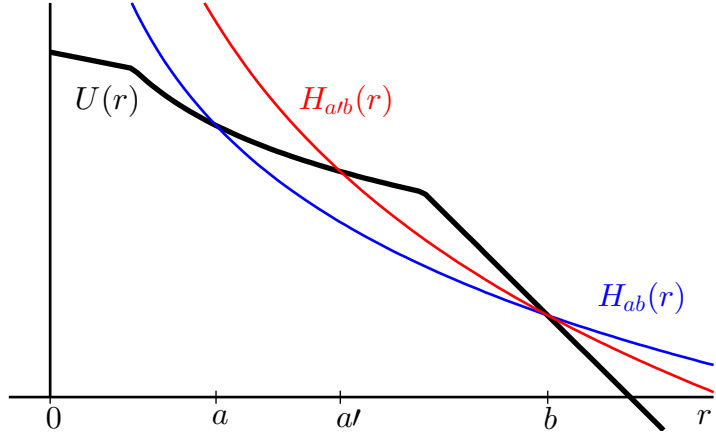}%
\caption{Proof of (3). \textcolor{red}{\(H_{a'b}\)} \(\geq\) \textcolor{blue}{\(H_{ab}\)} on \((0,b)\) so \textcolor{red}{\(H_{a'b}'(b)\)} \(\leq\) \textcolor{blue}{\(H_{ab}'(b)\)}.}%
\label{figradsup2}%
\end{figure}

\begin{proof}[Proof of (3)]
Let \(0<a<a'<b\). By the comparison principle,
\[H_{a'b}(a') = U(a')\geq H_{ab}(a').\]
Since \(H_{a'b}(b) = H_{ab}(b)\), the two functions are either identical or \(H_{a'b}>H_{ab}\) on \((0,b)\). It follows that
\begin{align*}
C_{ab}W_{n,p}'(b) &= H_{ab}'(b)\\
                  &= \lim_{\epsilon\to0^+}\frac{H_{ab}(b-\epsilon) - H_{ab}(b)}{-\epsilon}\\
									&\geq \lim_{\epsilon\to0^+}\frac{H_{a'b}(b-\epsilon) - H_{a'b}(b)}{-\epsilon}\\
									&= H_{a'b}'(b)\\
									&= C_{a'b}W_{n,p}'(b)
\end{align*}
and \(C_{ab} \leq C_{a'b}\) since \(W_{n,p}'(b)<0\). The proof for \(b<c'<c\) is symmetric.
\end{proof}

\begin{proof}[Proof of (4)]
The claims in (4) are immediate consequences of (2) and (3).
\end{proof}

We are now able to reveal a crucial fact about radial \(p\)-superharmonic functions: They have a smooth \(p\)-harmonic test function touching from \emph{above} at \emph{every} finite value.\footnote{This desirable property of having a test function from above is, unfortunately, not possessed by every \(p\)-superharmonic function, not even by the dominative ones. Although not proven here, the series
\[V(x) := \sum_{i=1}^\infty\frac{c_i}{|x-\xi/i|^\frac{n-p}{p-1}},\qquad c_i = \frac{1}{i^\frac{n-p}{p-1}2^i},\;|\xi| = 1,\; 2\leq p<n,\]
gives a counter example at \(x=0\). The unboundedness is \emph{not} the issue since the same could have been said about the function \(\min\{V,2\}\). On the other hand it is interesting to note that, in the case \(p=\infty\), \emph{every} dominative supersolution is touched from above by planes, i.e. 1-cylindrical fundamental solutions.}

\begin{lemma}\label{testabove}
\(2\leq p\leq\infty\). Let \(u\) be a radial \(p\)-superharmonic function and let \(x_0\in\mathbb{R}^n\). If \(u(x_0)<\infty\), then there exists a fundamental solution,
\[h(x) := C_1 w_{n,p}(x) + C_2, \qquad C_1 \geq 0,\]
touching \(u\) from above at \(x_0\):
\[h(x_0) = u(x_0)\qquad \text{and}\qquad h(x) \geq u(x)\;\text{ near \(x_0\)}.\]
\end{lemma}

\begin{proof}
If \(x_0=0\) and \(u\) is bounded at the origin, the constant function \(h(x) \equiv u(0)\) will do.

Write \(u(x) = U(|x|)\) and \(b := |x_0| > 0\). Let \(C_b\in [C_b^-, C_b^+]\) where the end-points of the, possibly singleton but non-empty, interval are defined in Lemma \ref{hlem}, (4). We claim that the scaled fundamental solution \(h(x) = H(|x|)\) given by
\[H(r) := C_b\Big[W_{n,p}(r) - W_{n,p}(b)\Big] + U(b)\]
touches \(u\) from above at \(x_0\).

Obviously, \(H(b) = U(b)\) and if \(0<a<b\), then \(C_{ab}\leq C_b\) by Lemma \ref{hlem} (3), (4) and
\begin{align*}
H(a) &= C_b\Big[W_{n,p}(a) - W_{n,p}(b)\Big] + U(b)\\
     &\geq C_{ab}\Big[W_{n,p}(a) - W_{n,p}(b)\Big] + U(b),\qquad W_{n,p}(a) - W_{n,p}(b) > 0,\\
		 &= H_{ab}(a) = U(a).
\end{align*}
If \(b<c\), then similarly \(C_b\leq C_{bc}\) and
\begin{align*}
H(c) &= C_b\Big[W_{n,p}(c) - W_{n,p}(b)\Big] + U(b)\\
     &\geq C_{bc}\Big[W_{n,p}(c) - W_{n,p}(b)\Big] + U(b),\qquad W_{n,p}(c) - W_{n,p}(b) < 0,\\
		 &= U(c).
\end{align*}
\end{proof}

Observe that \(H\) is uniquely determined if and only if \(U\) is differentiable at \(r=b\). That is, if and only if \(C_b^+ = C_b^-\).

With these new tools, we restate and prove the \emph{Radial Equivalence} (3) and the \emph{Sublinearity} (2) of Proposition \ref{D_p prop vis}.

\begin{proposition}[Radial Equivalence]\label{radeq vis}
Let \(2\leq p\leq\infty\). If \(u\) is a radial \(p\)-superharmonic function in \(\mathbb{R}^n\), then
\[\mathcal{D}_p u \leq 0\qquad\text{in }\mathbb{R}^n\]
in the viscosity sense.
\end{proposition}

\begin{proof}
Let \(x_0\in\mathbb{R}^n\) and assume \(u\) has a test function \(\phi\) from below at \(x_0\). We need to show that \(\mathcal{D}_p\phi(x_0)\leq 0\). Obviously, \(u(x_0)<\infty\). By Lemma \ref{testabove} there is a scaled fundamental solution \(h\) touching \(u\) from above at \(x_0\).
Thus, for \(x\) close to \(x_0\),
\[\phi(x)\leq u(x)\leq h(x),\qquad \phi(x_0) = u(x_0) = h(x_0)\]
which implies the Hessian matrix inequality \(\mathcal{H}\phi(x_0)\leq \mathcal{H}h(x_0)\).
It follows that
\[\mathcal{D}_p\phi(x_0) \leq \mathcal{D}_p h(x_0) = 0\]
from the fact that \(\mathcal{D}_p\) is degenerate elliptic and from the smooth case \emph{Cylindrical Equivalence}.
\end{proof}

\begin{proposition}[Sublinearity]\label{sublin}
\(2\leq p\leq\infty\). Assume \(\mathcal{D}_p u\leq 0\) and \(\mathcal{D}_p v\leq 0\) in the viscosity sense in \(\mathbb{R}^n\) where \(u\) is radial. Then
\[\mathcal{D}_p[u + v] \leq 0\]
in the viscosity sense in \(\mathbb{R}^n\).
\end{proposition}

\begin{proof}
Let \(x_0\in\mathbb{R}^n\) and assume \(\phi\) is a test function to \(u+v\) from below at \(x_0\).
Clearly, \(u(x_0)<\infty\) since otherwise \(u+v=\infty\) at \(x_0\) and there would be no test function there.
Also, \(u\) is \(p\)-superharmonic by the \emph{Domination} (1) of Proposition \ref{D_p prop vis}. Hence, by Lemma \ref{testabove}, there exists a scaled fundamental solution \(h\) touching \(u\) from above at \(x_0\):
\[h(x_0) = u(x_0),\qquad u(x) \leq h(x)\qquad\text{near \(x_0\)}.\]
Again, \(\mathcal{D}_p h = 0\) by the smooth case \emph{Cylindrical Equivalence}.

Define \(\psi(x) := \phi(x) - h(x)\). Then \(\psi\) is \(C^2\) and
\[\psi(x_0) = u(x_0) + v(x_0) - u(x_0) = v(x_0)\]
and
\[\psi(x) \leq u(x) + v(x) - u(x) = v(x)\]
near \(x_0\),
so \(\psi\) is a test function for \(v\) from below at \(x_0\). This means that \(\mathcal{D}_p\psi(x_0) \leq 0\), and it follows that
\[\mathcal{D}_p\phi = \mathcal{D}_p[\psi + h] \leq \mathcal{D}_p\psi + \mathcal{D}_p h \leq 0\]
at \(x_0\) by the smooth case \emph{Sublinearity}. Hence,
\[\mathcal{D}_p[u + v] \leq 0\]
in the viscosity sense in \(\mathbb{R}^n\).
\end{proof}

The proof of Proposition \ref{D_p prop vis}, and hence Theorem \ref{mainthm2}, is now completed.

\paragraph{Acknowledgements:}
I thank Peter Lindqvist for useful discussions and for naming the operator.
I thank Fredrik Arbo Høeg for checking calculations.
Also, I thank Juan Manfredi for pointing out the work \cite{MR2817413}.

%\begin{acknowledgement}
%I thank Peter Lindqvist for useful discussions and for naming the operator.
%I thank Fredrik Arbo Høeg for checking calculations.
%Also, I thank Juan Manfredi for pointing out the work \cite{MR2817413}.
%\end{acknowledgement}

\bibliographystyle{alpha}
\bibliography{C:/Users/Karl_K/Documents/PhD/references}

\begin{thebibliography}{LMS00}

\bibitem[Bru17]{Brustad201723}
Karl~K. Brustad.
\newblock Superposition in the \(p\)-{L}aplace equation.
\newblock {\em Nonlinear Analysis}, 158:23 -- 31, 2017.

\bibitem[CIL92]{MR1118699}
Michael~G. Crandall, Hitoshi Ishii, and Pierre-Louis Lions.
\newblock User's guide to viscosity solutions of second order partial
  differential equations.
\newblock {\em Bull. Amer. Math. Soc. (N.S.)}, 27(1):1--67, 1992.

\bibitem[CR15]{Cecil2015}
Thomas~E. Cecil and Patrick~J. Ryan.
\newblock {\em Isoparametric Hypersurfaces}, pages 85--184.
\newblock Springer New York, New York, NY, 2015.

\bibitem[CZ03]{MR1928087}
Michael~G. Crandall and Jianying Zhang.
\newblock Another way to say harmonic.
\newblock {\em Trans. Amer. Math. Soc.}, 355(1):241--263, 2003.

\bibitem[GT12]{MR2914612}
Nicola Garofalo and Jeremy~T. Tyson.
\newblock Riesz potentials and {$p$}-superharmonic functions in {L}ie groups of
  {H}eisenberg type.
\newblock {\em Bull. Lond. Math. Soc.}, 44(2):353--366, 2012.

\bibitem[JLM01]{MR1871417}
Petri Juutinen, Peter Lindqvist, and Juan~J. Manfredi.
\newblock On the equivalence of viscosity solutions and weak solutions for a
  quasi-linear equation.
\newblock {\em SIAM J. Math. Anal.}, 33(3):699--717, 2001.

\bibitem[Kat15]{MR3289084}
Nikos Katzourakis.
\newblock {\em An introduction to viscosity solutions for fully nonlinear {PDE}
  with applications to calculus of variations in {$L^\infty$}}.
\newblock SpringerBriefs in Mathematics. Springer, Cham, 2015.

\bibitem[Koi04]{MR2084272}
Shigeaki Koike.
\newblock {\em A beginner's guide to the theory of viscosity solutions},
  volume~13 of {\em MSJ Memoirs}.
\newblock Mathematical Society of Japan, Tokyo, 2004.

\bibitem[Li04]{MR2128299}
Song-Ying Li.
\newblock On the {D}irichlet problems for symmetric function equations of the
  eigenvalues of the complex {H}essian.
\newblock {\em Asian J. Math.}, 8(1):87--106, 2004.

\bibitem[Lin86]{Lindqvist1986}
Peter Lindqvist.
\newblock On the definition and properties of p-superharmonic functions.
\newblock {\em Journal für die reine und angewandte Mathematik}, 365:67--79,
  1986.

\bibitem[Lin16]{MR3467690}
Peter Lindqvist.
\newblock {\em Notes on the infinity {L}aplace equation}.
\newblock SpringerBriefs in Mathematics. BCAM Basque Center for Applied
  Mathematics, Bilbao; Springer, [Cham], 2016.

\bibitem[LM08]{MR2350398}
Peter Lindqvist and Juan~J. Manfredi.
\newblock Note on a remarkable superposition for a nonlinear equation.
\newblock {\em Proc. Amer. Math. Soc.}, 136(1):133--140, 2008.

\bibitem[LMS00]{Lindqvist2000}
Peter Lindqvist, Juan Manfredi, and Eero Saksman.
\newblock Superharmonicity of nonlinear ground states.
\newblock {\em Revista Matemática Iberoamericana}, 16(1):17--28, 2000.

\bibitem[OS11]{MR2817413}
Adam~M. Oberman and Luis Silvestre.
\newblock The {D}irichlet problem for the convex envelope.
\newblock {\em Trans. Amer. Math. Soc.}, 363(11):5871--5886, 2011.

\bibitem[Tho00]{Thorbergsson2000963}
Gudlaugur Thorbergsson.
\newblock Chapter 10 - a survey on isoparametric hypersurfaces and their
  generalizations.
\newblock In Franki~J.E. Dillen and Leopold~C.A. Verstraelen, editors, {\em
  Handbook of Differential Geometry}, volume~1 of {\em Handbook of Differential
  Geometry}, pages 963 -- 995. North-Holland, 2000.

\bibitem[Wan87]{MR901710}
Qi~Ming Wang.
\newblock Isoparametric functions on {R}iemannian manifolds. {I}.
\newblock {\em Math. Ann.}, 277(4):639--646, 1987.

\end{thebibliography}

\end{document}